\documentclass[12pt]{article}
\usepackage[numbers,sort&compress]{natbib}

\usepackage{amssymb} 
\usepackage{amsmath} 

\newcommand{\npmatrix}[1]{\left( \begin{matrix} #1 \end{matrix} \right)}

\DeclareMathOperator{\trace}{trace}
\newcommand{\R}{\mathbb{R}}
\newcommand{\F}{\mathbb{F}}

\newtheorem{proposition}{Proposition}[section]
\newtheorem{theorem}{Theorem}[section]
\newtheorem{example}{Example}[section]

\newtheorem{remark}{Remark}[section]
\newtheorem{definition}{Definition}[section]
\newtheorem{corollary}{Corollary}[section]
\newtheorem{lemma}{Lemma}[section]
\newenvironment{proof}{\textbf{Proof.}}{\qquad $\Box$ \bigskip }

\newenvironment{mainproof}{\textbf{Proof of Theorem \ref{main}.}}{\qquad $\Box$ \bigskip }

\begin{document}

\title{Power series with positive coefficients arising from the characteristic polynomials of positive matrices  \thanks{This work was supported by Science Foundation Ireland under Grant 11/RFP.1/MTH/3157}}

\author{Thomas J. Laffey,\ \ University College Dublin \and Raphael Loewy,\ \ Technion  \and Helena \v Smigoc,\ \ University College Dublin}

\date{}

\maketitle

\begin{abstract}
Let $A$ be an $n \times n$ (entrywise) positive matrix and let $f(t)=\det(I-t A)$. We prove the surprising result that there always exists a positive integer $N$ such that the formal power series expansion of $1-f(t)^{1/N}$ around $t=0$ has positive coefficients. 

\parindent=14 pt


{\bf Mathematics Subject Classication:} 15B48, 15A18,  37B10, 30B10
\end{abstract}

\section{Introduction}

Questions about power series are classical \cite{MR0034854,MR0027306,MR1544949,MR1544805,MR0030620,MR1575394}, and properties of coefficients of power series is one of the central topics in this context. 
For example, questions about the signs of the coefficients of reciprocal power series \cite{MR1544949,MR0064890,MR2869264} have found application in the study of renewal sequences, which are frequently applied in probability theory \cite{MR972601,MR0271658,MR0224175}. 

The problem of deciding whether a given function has a power series expansion with all its coefficients positive is of seemingly elementary nature, but can be surprisingly difficult. For example, a conjecture of H.~Lewy and K.~Friedrichs, which arose form their work on difference approximations to the wave equation, is that the rational function
 $$\frac{1}{(1-x)(1-y)+(1-y)(1-z)+(1-z)(1-x)}=
\sum_{k,m,n \geq 0}a(k,m,n)x^ky^mz^n$$  
has $a(k,m,n)>0$.  This was first proved by Szeg\"o \cite{MR1545428} using involved arguments on Bessel functions. This motivated a series of papers using a range of different methods \cite{MR1545429,MR0301441,MR2424632,Kauers07computeralgebra,MR2455705}. Despite a considerable body of work, a seemingly simple question about the positivity of the coefficients of the multivariate series expansion about the origin of 
 $$\frac{1}{1-x-y-z-w+\frac{2}{3}(xy+xz+xw+yz+yw+zw)}$$
 remains open since 1972 \cite{MR0301441}.

Apart from the variations of Szeg\"o's functions mentioned above, we know of very little literature on general families of functions with positive coefficients.   
In this paper we present a new family of multivariate functions with positive coefficients. This family arises from the determinants of structured matrices of the form
\begin{equation}\label{Sn}
X_n=\npmatrix{x_1 & 1 & 0 &  0 & \ldots & 0 \\
                            x_2 & x_1 & 2 & 0 & \ddots & 0 \\
                            x_3 & x_2 & x_1 & 3 & \ddots & \vdots \\ 
                            \vdots & \ddots & \ddots & \ddots & \ddots &0 \\
                            x_{n-1} &  & x_3  & x_2 & x_1 & (n-1) \\
                            x_n  & x_{n-1} & \ldots & x_3 & x_2 & x_1}.
\end{equation}
We show that multivariate expansion of $1-\det(I_n-tX_n)^{1/n}$ in $t, x_1,\ldots,x_n$ has positive coefficients. Our method depends on the matrix representation of the problem and uses tools from matrix theory, such as the trace vector, in an essential way. Despite looking at different methods, from outside matrix analysis, we were not able to find an alternative proof of this result. 

Matrices of the type (\ref{Sn}) were used in a constructive approach to the Boyle-Handelman result \cite{MR1097240} characterizing the nonzero spectra of entry-wise nonnegative matrices \cite{Laffey20121701}. This gave rise to the following result. Let $A$ be an $n \times n$ (entrywise) positive matrix and let $f(t)=\det(I_n-t A)$. We prove the surprising result that there always exists a positive integer $N$ such that the formal power series expansion of $1-f(t)^{1/N}$ around $t=0$ has positive coefficients.

\section{Short overview of the NIEP}\label{NIEP}

Spectral properties of (entrywise) nonnegative matrices have been studied widely in recent years. 
One of the central problems in this area is the \emph{nonnegative inverse eigenvalue problem (NIEP)}; the problem  of finding
necessary and sufficient conditions in order that a list of complex
numbers $\sigma=(\lambda_1,\lambda_2,\ldots,\lambda_n)$ be the spectrum of
an entrywise nonnegative matrix. In this section we recall some results on the NIEP that will provide a setting for the rest of the paper.  

If a list of complex numbers $\sigma$ is the spectrum of some nonnegative matrix $A$, we say that $\sigma$ is \emph{realizable} and that $A$ is \emph{a realizing matrix}  for $\sigma$. A classical result of Perron and Frobenius tells us that a realizable list of complex numbers has to contain a nonnegative real number that is greater than
or equal to the absolute value of any other number in the list. This number is
called the \emph{Perron eigenvalue}. Since for any nonnegative matrix $A$ the trace of
$A^k,$ $k=1,2,\ldots,$ is nonnegative, the following conditions must hold:
 $$s_k(\sigma)=\lambda_1^k+\lambda_2^k+\ldots+\lambda_n^k \geq 0.$$
In \cite{MR0480563} and independently in \cite{MR618581}, JLL inequalities were
proved. They tell us that a realizable list of $n$ complex numbers $\sigma$
satisfies
  \begin{equation}\label{JLL}
 n^{k-1}s_{km}(\sigma)\geq s_{m}^k(\sigma)
 \end{equation}
 for all positive integers $k$ and $m.$ Necessary conditions that we mentioned above are sufficient only in the case where $n \leq 3,$ \cite{MR0480563}. A solution for $n=4$ appears in the PhD thesis of Meehan \cite{Meehan}, and a different solution in terms of the coefficients of the characteristic polynomial is given in \cite{MR2350690}. 

While we are far from the complete solution to the NIEP, its variation, the problem of characterizing the nonzero spectra of nonnegative matrices, was solved by Boyle and Handelman in \cite{MR1097240}. We say that a list
of $n$ complex numbers $\sigma$ is the \emph{nonzero spectrum} of a nonnegative matrix, if
there exists a nonnegative integer $N$ such that $\sigma$ together with $N$
zeros added to it, is the spectrum of some $(n+N)\times (n+N)$
nonnegative matrix. Boyle and Handelman proved the following result.

\begin{theorem}(\cite{MR1097240}) \label{BH91}
A list of complex numbers $\sigma=(\lambda_1,\lambda_2,\ldots,\lambda_n)$ is
the nonzero spectrum of some nonnegative matrix if  the following conditions
hold:
\begin{enumerate}
\item $\sigma$ has a Perron eigenvalue $\lambda_1$ with $\lambda_1 >
|\lambda_i|$ for $i=2,\ldots,n.$
\item $\sigma$ is closed under complex conjugation.
\item For all positive integers $m$,
 $$s_m(\sigma)\geq 0,$$
and
 $s_m(\sigma)>0$ implies $s_{mk}(\sigma)>0$ for all positive integers $k.$
\end{enumerate}
\end{theorem}

Under the assumptions of the theorem, a realizing matrix can be chosen to be primitive, and in this case the conditions are necessary and sufficient. See Friedland \cite{Friedland:2009fk} for an extension to the irreducible case. 

The proof of Theorem \ref{BH91} in \cite{MR1097240} is not algorithmic and does not provide a bound on the minimal number $N$ of zeros required for realizability. 
Laffey \cite{Laffey20121701} found a constructive approach to the Boyle and Handelman theorem using the matrix of the type (\ref{Sn}).
He proved the following result.

\begin{theorem}[\cite{Laffey20121701}]\label{BHL}
Let $\sigma=(\lambda_1,\lambda_2,\ldots,\lambda_n)$ be a list of complex numbers  that satisfy
\begin{enumerate}
\item $\lambda_1>|\lambda_j|$ for $j=2,\ldots,n.$
\item $\sigma$ is closed under complex conjugation.
\item $s_1(\sigma) \geq 0$ and $s_m(\sigma)>0$ for all $m \geq 2.$
\end{enumerate}
Then there exists a positive integer $N$ such that $\sigma$ with $N$ zeros added is the spectrum of a nonnegative matrix $X_{n+N}$ defined in (\ref{Sn}) for some nonnegative real numbers $x_i,$ $i=1,2,\ldots,n+N$.
\end{theorem}

Furthermore, a bound on the minimal number of zeros $N$ needed to be added is presented in \cite{Laffey20121701}.
One of the main observations needed to prove Theorem \ref{BHL} is the following proposition.

\begin{proposition}[\cite{Laffey20121701}]
Let $(\mu_1,\mu_2,\ldots,\mu_n)$ be a list of complex numbers and let 
 \begin{equation}\label{qq} 
  q(x):=\prod_{j=1}^n (x-\mu_j)=x^n+q_1 x^{n-1}+\ldots + q_n.
 \end{equation} 
If we put $x_k:=\sum_{j=1}^n \mu_j^k$ in (\ref{Sn}), then:  
\begin{equation}\label{charpoly}
\det(xI_n- X_n)=x^n+nq_1x^{n-1}+n(n-1)q_2 x^{n-2}+\ldots+n!q_n.
\end{equation}
\end{proposition}

Consider a polynomial:
 $$F(x)=(x-\lambda_1) (x-\lambda_2)\ldots (x-\lambda_n)=x^n+p_1x^{n-1}+\ldots+p_n.$$
If the coefficients $p_i \leq 0$, the companion matrix $C(F)$ of $F(x)$ has nonnegative entries and realizes $(\lambda_1, \lambda_2, \ldots, \lambda_n)$. While this condition on $F(x)$ is quite restrictive, it does occur in some interesting situations. For example, the first major result on the NIEP was obtained by Suleimanova \cite{MR0030496}, who proved that if $\sigma$ consists only of real numbers and
 $\lambda_1>0$ and $\lambda_i \leq 0,$
$i=2,\ldots,n,$ then $\sigma$ is realizable if and only if 
$$s_1(\sigma)=\lambda_1+\lambda_2+\ldots+\lambda_n\geq 0.$$ 
Friedland \cite{MR492634} obtained an elegant proof of this result by establishing that the companion matrix of $F(x)$ is nonnegative in this case. 

Laffey and \v Smigoc \cite{MR2232926} proved that if all elements of $\sigma$ other than its Perron element have non-positive real parts, then $\sigma$ is realizable if and only if $\sigma$ is realizable by a matrix of the form $C+\alpha I,$ where $\alpha \geq 0$ and $C$ is a nonnegative companion matrix. They gave a complete easy to verify characterization of such $\sigma$.. 

More generally, they also considered realizations using matrices of the form
\begin{equation}\label{CT}
\npmatrix{C(f_1) & N_1 & 0 & \ldots & 0 \\
                        0 & C(f_2) & N_2 & \ddots & \vdots \\
                       \vdots & \ddots & \ddots & \ddots & 0 \\
                       0 & \ldots & 0 & \ddots & N_{k-1} \\
                       R_1& \ldots & R_{k-2} & R_{k-1} & C(f_k)},
                       \end{equation}
                       
 where 
 \begin{itemize}
\item $C(f_i)$ denotes the companion matrix of polynomial $f_i$ for $i=1,2,\ldots, k$,
\item $N_i$ denotes a matrix of an appropriate size that has the element in the lower left corner equal to $1$ and all other elements equal to zero for $i=1,\ldots,k-1,$ 
\item $R_i$ denotes a matrix of an appropriate size whose elements are all equal to zero except possibly those on the last row for $i=1,\ldots,k-1$.
\end{itemize}
They have used such matrices to improve known bounds for realizability in the NIEP \cite{MR2420980,MR2742335}. They have also observed that a large class of spectra can be realized in this way. 

In establishing a conjecture of Boyle and Handelman on the realizability of spectra by nonnegative integer matrices Kim, Ormes and Roush  \cite{MR1775737} introduced a formal factorization of 
 $$f(t)=\prod_{j=1}^n (1-\lambda_j t)$$
in the form 
 $$f(t)=g_1(t)g_2(t)\ldots g_k(t)r(t),$$
 where $1-g_j(t)$ are polynomials with nonnegative coefficients and $1-r(t)$
is a formal power series with nonnegative coefficients, in order to get realizations over the semiring $\mathbb{Z}_+[t]$ of polynomials with nonnegative integer coefficients. This enabled them to deduce realizations over the nonnegative integers. It turns out that one can obtain realizations of block-companion type (\ref{CT}) above in this case.  

The results mentioned above explain how the nonnegativity of coefficients of certain polynomials related to the characteristic polynomial $F(x)$ can be used effectively in the NIEP. This led to considering finding such results for general realizable $\sigma$. In particular, to the question: if 
 $$f(t)=(1-\lambda_1 t)(1-\lambda_2 t)\ldots (1-\lambda_n t),$$
 where $(\lambda_1,\lambda_2,\ldots,\lambda_n)$ is a realizable list, does $1-f(t)^{1/n}$ have nonnegative coefficients? The answer is "No," in general. However, we observed that in all cases we tested, $1-f(t)^{1/N}$ had nonnegative coefficients for all sufficiently large positive integers $N$. Thus we were led to a conjecture that this always occurs. In this paper we prove the conjecture for lists $\sigma$ having a Perron element $\rho=\lambda_1 > |\lambda_j|$, $j=2,\ldots, n$, and having its Newton power sums $s_k(\sigma)>0$ for $k \geq 2$. The proof is quite indirect and involves the analysis of the matrices $X_n$ above. Several interesting properties of these matrices are uncovered and used in the proof.

\section{Power series with positive coefficients}\label{pswpc} 

Let $\sigma=(\lambda_1,\lambda_2,\ldots,\lambda_n)$ be a list of complex numbers.
Let us define 
\begin{equation}\label{Ff}
F(x)=\prod_{i=1}^n(x-\lambda_i) \text{ and } f(t)=t^nF(\frac{1}{t}).
\end{equation}
Notice that $f(t)$ depends only on the nonzero elements of $\sigma.$ Observe that the companion matrix of $F(x)$ is nonnegative if and only if $1-f(t)$ has nonnegative coefficients.  While this is true only for some realizable lists $\sigma$, we have noticed that in many cases $1-f(t)^{1/n}$ has positive coefficients. In particular, the following result was conjectured and proved by Laffey \cite{LaffeyPos}. Different proofs have been found by Aharonov \cite{AharonovPos}, Holland \cite{HollandPos} and Kova\v cec \cite{KovacecPos}.

\begin{theorem}[\cite{LaffeyPos}]\label{positive}
Let $\lambda_1,\lambda_2,\ldots, \lambda_n$ be positive real numbers and let  
 $$f(t)=\left(\prod_{i=1}^n (1-\lambda_i t)\right)^{1/n}.$$
Then $1-f(t)$ has nonnegative coefficients.  
\end{theorem}

Let $(\lambda_1,\lambda_2,\ldots,\lambda_n)$ be a realizable list of complex numbers and let $$f(t)=\prod_{i=1}^n(1-\lambda_i t).$$ The following examples show that $1-f(t)^{1/n}$ need not have nonnegative coefficients.

\begin{example}
Let $$\sigma=(1,\frac{9}{10},-\frac{9}{10}) \text{ and }f(t)=(1-t)(1-\frac{9}{10}t)(1+\frac{9}{10}t).$$ The 
list $\sigma$ is realizable, but the power series expansion of $1-f(t)^{1/3}$  does not have all its coefficients nonnegative. The power series expansion of $1-f(t)^{1/4}$ has positive coefficients. 
\end{example}

\begin{example}
Let $$\sigma=(1,\frac{99}{100},-\frac{99}{100})\text{ and }f(t)=(1-t)(1-\frac{99}{100}t)(1+\frac{99}{100}t).$$ The 
list $\sigma$ is realizable, but the power series expansion of $1-f(t)^{1/k},$ $k=3,4,5$,  does not have all its coefficients nonnegative. The power series expansion of $1-f(t)^{1/6}$ has positive coefficients. 
\end{example}

Let $F(x)$ be the characteristic polynomial of an $n \times n$ positive matrix $A$ and let $f(t)$ be as defined in (\ref{Ff}). In this work we will prove that there always exists a positive integer $N$ such that $1-f(t)^{1/N}$ has nonnegative coefficients. 

\begin{definition}
Let $F(t)=\sum_{t=0}^{\infty} \Gamma_i t^i$ be a polynomial in $t$ or a formal power series expansion of $F(t)$ around zero, where  the coefficient $\Gamma_j$ of $t^j$ is a polynomial in $x_1, x_2, \ldots,x_n$ for all powers $t^j$: 
$$\Gamma_j=\sum_{(\alpha_1, \alpha_2,\ldots \alpha_n) \in (\mathbb{N}\cup \{0\})^n} x_1^{\alpha_1}x_2^{\alpha_2}\ldots x_n^{\alpha_n}\beta(\alpha_1,\alpha_2,\ldots,\alpha_n),$$ $\beta(\alpha_1,\alpha_2,\ldots,\alpha_n) \in \mathbb{\R}$. We say that $F(t)$ is \emph{monomially positive} if the coefficients $\beta(\alpha_1,\alpha_2,\ldots,\alpha_n)$ are nonnegative.
\end{definition}

\begin{example}
The polynomial $p_0(t)=t+(x_1^2-x_2)^2t^4$ is not monomially positive.
\end{example}

Let $X_n$ be the matrix defined in (\ref{Sn}) and let us define $$F_n(x)=\det(x I_n-X_n).$$ Furthermore, let $f_0(t)=1$, and for $n \geq 1$
$$f_n(t)=\det(I_n-t X_n)$$
be the polynomial obtained from $F_n(x)$ by $f_n(t)=t^nF_n(\frac{1}{t})$.

Notice that $f_n(t)$ is a polynomial in $t$ of degree $n$
$$f_n(t)=1-\sum_{j=1}^n \hat{\gamma}_j t^j,$$
where the coefficients $\hat{\gamma}_j$ of $t$ are in turn multivariable polynomials in $x_1, x_2, \ldots,x_n$:
$$\hat{\gamma}_j=\sum_{\alpha_1+2 \alpha_2+\ldots + n \alpha_n=j} x_1^{\alpha_1}x_2^{\alpha_2}\ldots x_n^{\alpha_n}\hat{\beta}(\alpha_1,\alpha_2,\ldots,\alpha_n) \in \R[x_1,x_2,\ldots,x_n].$$
To explain the sum under the summation sign in the above formula we note that in the expansion of the determinant of $(I_n-t X_n)$ every occurrence of $x_j$ is accompanied with $t^j$, so we can think of $x_j$ as being associated with the weight $j$. 
Now let us look at the formal power series expansion of $(f_n(t))^{1/n}:$
\begin{align*} 
 (f_n(t))^{1/n}&=1-\sum_{j=1}^{\infty} \gamma_j t^j,\\
\gamma_j&=\sum_{\alpha_1+2 \alpha_2+\ldots + n \alpha_n=j} x_1^{\alpha_1}x_2^{\alpha_2}\ldots x_n^{\alpha_n}\beta(\alpha_1,\alpha_2,\ldots,\alpha_n).
\end{align*}
We will show that the coefficients $\beta(\alpha_1,\alpha_2,\ldots,\alpha_n)$ are nonnegative; or equivalently
that $1-(f_n(t))^{1/n}$ is monomially positive.

\begin{theorem}\label{main}
Let $f_n(t)=\det(I_n-t X_n),$ where $X_n$ is defined in (\ref{Sn}). Then
$$1-(f_n(t))^{1/n}$$ is monomially positive. 
\end{theorem}

Using Theorem \ref{BHL} and Theorem \ref{main}, we can prove the main result of this paper.

\begin{theorem}\label{NIEPappl}
Let $f(t)=\prod_{i=1}^n(1- \lambda_i t)$ be a polynomial that satisfies:
\begin{enumerate}
 \item $s_k=\sum_{i=1}^n \lambda_i^k >0$ for $k=1,2,\ldots$ 
 \item $\lambda_1 > |\lambda_i|$ for $i=2,3,\ldots,n.$ 
\end{enumerate} 
 Then there exists a positive integer $N_0$ so that $1-f(t)^{1/N}$ has positive coefficients for all $N > N_0.$ 
\end{theorem}

\begin{proof}
Let $f(t)$ be a polynomial satisfying the assumptions of the theorem. By Theorem \ref{BHL} there exists a positive integer $N_0$ so that $$f(t)=\det(I_{N_0}-t X_{N_0}),$$ for a nonnegative matrix $X_{N_0}$ of the form (\ref{Sn}). Now Theorem \ref{main} tells us that $1-f(t)^{1/N_0}$ has nonnegative coefficients. We have
 $$f(t)^{1/N_0}=1-\sum_{j=1}^{\infty} \gamma_j t^j,$$
 where $\gamma_j \geq 0$ and $\gamma_1=\frac{s_1}{N_0}>0$.

Since the power series expansion of $1-(1-z)^{\alpha}$ around $z=0$ has positive coefficients for $\alpha \in (0,1)$ it follows that for $N>N_0$
 $$1-f(t)^{1/N}=1-(f(t)^{1/N_0})^{N_0/N}$$ 
 has positive coefficients.
 \end{proof}

\begin{corollary}
Let $A$ be an $n \times n$ positive matrix  and let $f(t)=\det(I_n-t A).$ Then there exists a positive integer $N_0$ so that $1-f(t)^{1/N}$ has positive coefficients for all $N \geq N_0.$ 
\end{corollary}

\begin{example}
The power series expansion of $1-((1-t^2)(1-t^3))^{1/N}$ around $t=0$ does not have nonnegative coefficients for any positive integer $N$. For example, it is easy to check that the coefficient of $t^5$ is always negative. This example shows that the existence of the Perron root $\lambda_1>|\lambda_i|,$ $i=2,3,\ldots,n$, is a necessary assumption in Theorem \ref{NIEPappl}. 
\end{example}

To prove a partial converse of Theorem \ref{NIEPappl} we need the following well known lemma. 

\begin{lemma}\label{lem1}
Let $\F$ be a field of characteristic zero, $A \in M_n(\F)$, $f(t)=\det (I_n-t A)$ and $s_k=\trace (A^k)$. Then:
 $$\frac{f'(t)}{f(t)}=-\sum_{j=1}^{\infty}s_jt^{j-1}.$$
\end{lemma}

\begin{proposition}
Let $f(t)=\prod_{i=1}^n(1-\lambda_i t)$ and let $s_k=\sum_{i=1}^n \lambda_i^k.$ If there exists a positive integer $N$ such that $1-f(t)^{1/N}$ has positive coefficients, then $s_k>0$ for all positive integers $k$.
\end{proposition}

\begin{proof}
Let $h(t)=f(t)^{1/N}$, where $N$ is chosen so that $1-h(t)$ has positive coefficients. Then $-h'(t)$ and $\frac{1}{h(t)}$ both have positive coefficients. This implies that $$\frac{h'(t)}{h(t)}$$
has negative coefficients. 
On the other hand: 
$$\frac{h'(t)}{h(t)}=\frac{1}{N}\frac{f'(t)}{f(t)}=-\sum_{j=1}^{\infty} s_jt^{j-1}$$
by Lemma \ref{lem1}.
\end{proof}

\section{Special cases}\label{sc}

Before we give a complete proof of Theorem \ref{main} in Section \ref{Mpps} we look at some examples and special cases. To illustrate our problem let us look at a short proof of Theorem \ref{main} in the case when $n=2.$

\begin{example}
In the case $n=2$ we have
$$
 X_2=\npmatrix{x_1 & 1 \\ x_2 & x_1} \text{ and }
 f_2(t)=1 - 2 t x_1 + t^2 (x_1^2 - x_2).
$$  
Now we have
\begin{align*}
f_2(t)^{1/2}&=((1-x_1t)^2-x_2t^2)^{1/2} \\
                   &=(1-x_1t)\Big(1-\frac{x_2t^2}{(1-x_1t)^2}\Big)^{1/2} \\
                   &=(1-x_1 t) \left(1+\sum_{j=1}^{\infty} (-1)^j \binom{1/2}{j}\left(\frac{x_2t^2}{(1-x_1t)^2}\right)^j \right)\\
                    &= \left(1-x_1t-\sum_{j=1}^{\infty}\zeta_j\frac{(x_2t^2)^j}{(1-x_1t)^{2j-1}} \right), 
\end{align*}
 where $\zeta_j=(-1)^{j-1} \binom{1/2}{j}>0.$ Notice that $\frac{(x_2t^2)^j}{(1-x_1t)^{2j-1}}$ is monomially positive, which proves that 
  $$1-f_2(t)^{1/2}$$
is monomially positive.
\end{example}

Similar arguments may be used to prove Theorem \ref{main} in the case $n=3$ and $n=4,$ but we were unable to modify this argument to prove the general statement, and had to adopt a more circuitous route.

Theorem \ref{main} implies that $$1-(f_n(t))^{1/N}=1-g_{N,n}(t)$$ is monomially positive for all $N \geq n.$ However, no $N <n$ would give us the result. To see this we consider the case where $x_2=x_3=\ldots=x_n=0$ and $f_n(t)=(1-x_1t)^n.$ Clearly, the conclusion of the Theorem holds, since $(f_n(t))^{1/n}=1-x_1t$, but is not true for  $(f_n(t))^{1/N}$ for any $N<n.$

Now we consider the special case where $x_1=0,$ $x_3=\ldots=x_n=0.$ Without loss of generality we may assume that $x_2=1.$ Let us denote 
 $$X_{n,2}=\npmatrix{0 & 1 & 0 &  0 & \ldots & 0 \\
                            1 & 0 & 2 & 0 & \ldots & 0 \\
                            0 & 1& 0 & 3 & \ddots & \vdots \\ 
                            \vdots & \ddots & \ddots & \ddots & \ddots &0 \\
                            0 & \ldots & 0 & 1 & 0 & (n-1) \\
                            0  & \ldots & 0& 0 & 1 & 0}.
$$
and 
$f_{n,2}(t)=\det(I_n-t X_{n,2}).$

\begin{proposition}
$1-(f_{n,2}(t))^{1/{\lfloor \frac{n}{2}\rfloor}}$ has nonnegative coefficients.
\end{proposition}

\begin{proof}
First notice that $f_{n,2}(t)$ is a function of $t^2.$ We can write
 $$f_{n,2}(t)=\prod_{i=1}^{\lfloor \frac{n}{2}\rfloor}(1-a_it^2)$$
 for some complex numbers $a_i.$
Let $D_n$ be the diagonal matrix with the diagonal elements $$1,1,1/\sqrt{2},1/\sqrt{3!},\ldots,1/\sqrt{(n-1)!}.$$ Then $D_n^{-1}X_{n,2}D_n$ is a symmetric matrix, so all the roots of the polynomial $f_{n,2}(t)$ are real. This implies that $a_i> 0$ for $i=1,2,\ldots,\lfloor n/2\rfloor.$ Now we use Theorem \ref{positive} to finish the proof. 
\end{proof}

\section{The Trace Vector of $X_n$}

In \cite{MR2705278} the notion of the trace vector was introduced and a number of interesting results associated with it were proved. Pereira  \cite{MR2705278} showed that there exists a trace vector for every matrix  $A \in \mathbb{C}^{n \times n}$, however it may, in general, be difficult to find. We will show that the trace vector for $X_n$ is the standard basis vector $e_n$.

\begin{definition}
Let $\F$ be a field of characteristic $0$, $A \in M_n(\F)$ . Then 
 $$t_1(A)=\frac{1}{n}\trace(A)$$ 
 is called \emph{the normalized trace} of $A.$
\end{definition}

Over a general field $\F$ of characteristic $0$ we can define the trace vector in the following way.

\begin{definition}\label{DefTrace}
Let $\F$ be a field of characteristic $0$, $A \in M_n(\F)$ and $z \in \F^{n}.$ We say that $z$ is a \emph{trace vector} of $A$ if $$z^Tp(A)z=t_1(p(A))$$ for all polynomials $p.$ 
\end{definition}

\begin{remark}
Pereira \cite{MR2705278} defined a trace vector for
matrices over the complex field as in Definition \ref{DefTrace} except that the
transpose is replaced by * which means "conjugate complex transpose":
$$z^*p(A)z=t_1(p(A)).$$
He proved existence of a trace vector of his form for any $n \times n$ complex matrix, but we do not consider existence results for general matrices
here.
\end{remark}

We need a version of Theorem 2.5 of Pereira \cite{MR2705278} over general fields. 

\begin{theorem}[\cite{MR2705278}]\label{en}
Let $\F$ be a field of characteristic $0$,  $A \in M_n(\F),$ and let $A(n)\in M_{n-1}(\F)$ be the principal submatrix obtained from $A$ by deleting the $n$-th row and the $n$-th column. The standard basis vector $e_n$  is a trace vector of $A$ if and only if $$p_{A(n)}(x)=\frac{1}{n}\frac{d}{dx}p_A(x),$$
where $p_A$ and $p_{A(n)}$, respectively, denote the characteristic polynomials of $A$ and $A(n)$, and $\frac{d}{dx}$ stands for the derivative with respect to $x$.   
\end{theorem}

\begin{proof}
Pereira's proof of this result easily extends to general fields and we include it here for convenience. 

Formally we can write 
 $$(xI_n-A)^{-1}=\sum_{j=0}^{\infty} \frac{1}{x^{j+1}} A^j.$$
Let $\lambda_1,\lambda_2,\ldots,\lambda_n$ be the eigenvalues of $A$ in an appropriate extension field of $\F$.
 Let $e_n$ be a trace vector for $A$, then
\begin{align*}
\frac{p_{A(n)}(x)}{p_A(x)}&= e_n^T(xI_n-A)^{-1}e_n \\
                 &=\frac{1}{n}\sum_{j=0}^{\infty} \frac{1}{x^{j+1}}\trace (A^j) \\
                 &=\frac{1}{n}\trace((xI_n-A)^{-1}) \\
                 &=\frac{1}{n}\sum_{j=1}^n \frac{1}{x-\lambda_j} \\
                 &=\frac{1}{n}\frac{\frac{d}{dx}p_A(x)}{p_A(x)}.
 \end{align*}
Conversly, let $p_{A(n)}(x)=\frac{1}{n}\frac{d}{dx}p_A(x)$. Then
 \begin{align*}
 e_n^T(xI_n-A)^{-1}e_n&=\frac{p_{A(n)}(x)}{p_A(x)}\\
              &=\frac{1}{n}\frac{\frac{d}{dx}p_A(x)}{p_A(x)}\\
              &=\frac{1}{n}\sum_{j=1}^n \frac{1}{x-\lambda_j}\\
              &=\frac{1}{n}\trace((xI_n-A)^{-1}).
 \end{align*}
 
\end{proof}

Next we find recursive equations for $F_n(x)$ and $f_n(t).$

\begin{lemma}\label{expansion}
Let us define  $f_0(t):=1$  and  $F_0(x):=1$. Then:
\begin{enumerate}
\item
$F_n(x)=(x- x_1)F_{n-1}(x)-\sum_{i=2}^n\frac{(n-1)!}{(n-i)!}x_iF_{n-i}(x).$
\item $f_n(t)=(1-x_1t)f_{n-1}(t)-\sum_{i=2}^n\frac{(n-1)!}{(n-i)!}x_if_{n-i}(t)t^i.$
\end{enumerate}
\end{lemma}

\begin{proof}
We can prove the recursive equation for $F_n(x)$ by expanding  $\det(xI_n- X_n)$ along its last row.
The recursive relation for $f_n(t)$ then follows from the equality $f_n(t)=t^nF_n(\frac{1}{t})$.
\end{proof}

\begin{proposition}\label{diff1}
 $\frac{d}{dx}F_n(x)=nF_{n-1}(x)$
\end{proposition}
 
\begin{proof}
We will prove this proposition by induction on $n$. Let assume that 
\begin{equation}\label{hypo}
\frac{d}{dx}F_k(x)=kF_{k-1}(x)
\end{equation}
for $k=1,2,\ldots,n-1.$ 
Now we differentiate the recursive equation for $F_n(x)$ from Lemma \ref{expansion} and use 
(\ref{hypo}):
\begin{align*}
\frac{d}{dx}F_n(x)&=F_{n-1}(x)+(x- x_1)\frac{d}{dx}F_{n-1}(x)-\sum_{i=2}^n\tfrac{(n-1)!}{(n-i)!}x_i\frac{d}{dx}F_{n-i}(x)\\
&=F_{n-1}(x)+(x- x_1)(n-1)F_{n-2}(x)-\sum_{i=2}^{n-1}\tfrac{(n-1)!}{(n-i)!}x_i(n-i)F_{n-i-1}(x)\\
&=F_{n-1}(x)+(n-1)\left( (x- x_1)F_{n-2}(x)-\sum_{i=2}^{n-1}\tfrac{(n-2)!}{(n-1-i)!}x_iF_{n-i-1}(x)\right)\\
&=F_{n-1}(x)+(n-1)F_{n-1}(x),
\end{align*}
where the last equality follows form the recursive equation for $F_{n-1}(x).$
\end{proof} 
 
\begin{corollary}\label{trace}
The standard basis vector $e_n$  is a trace vector for $X_n.$
\end{corollary}

\begin{proof}
Corollary of Theorem \ref{en} and Proposition \ref{diff1}.
\end{proof}

\section{Monomially positive power series}\label{Mpps}

\begin{lemma}
The power series expansion of $\frac{f_{n-1}(t)}{f_n(t)}$ around $t=0$ is monomially positive for all positive integers  $n$.
\end{lemma}

\begin{proof}
If $n=1$ the statement is clear since $f_0(t)=1$ and $f_1(t)=1-x_1 t$. Suppose $n\geq 2$.
Partition $X_n$ in the following way:
 $$X_n=\npmatrix{X_{n-1} & u_n \\ v_n^T & x_1},$$
 where $u_n=(0,\ldots,0,n-1)^T$ and $v_n=(x_n,x_{n-1},\ldots,x_2).$
 Then
  \begin{align*}
  f_n(t)&=\det(I_n-t X_n) \\
           &=\det(I_{n-1}-t X_{n-1}) (1-x_1 t-v_n^T(I_{n-1}-tX_{n-1})^{-1}u_nt^2).
  \end{align*}
  So
  $$
 \frac{f_{n-1}(t)}{ f_n(t)}=\frac{1}{(1-x_1 t-v_n^T(I_{n-1}-tX_{n-1})^{-1}u_nt^2)}.
  $$
From the formal expansion $$(I_{n-1}-tX_{n-1})^{-1}=I_{n-1}+tX_{n-1}+t^2X_{n-1}^2+\ldots,$$
 we see that  $v_n^T(I_{n-1}-tX_{n-1})^{-1}u_n$ is monomially positive. 
We have
  $$ \frac{f_{n-1}(t)}{ f_n(t)}=\frac{1}{1-U_n(t)},$$
where $U_n(t)$ is monomially positive. This proves that $ \frac{f_{n-1}(t)}{ f_n(t)}$ is monomially positive. 
\end{proof}

\begin{corollary}\label{quot}
The power series expansion of $\frac{f_{k}(t)}{f_n(t)}$ around $t=0$ is monomially positive for $k =0,1,2,\ldots,n-1$. 
\end{corollary}

\begin{proof}
We write
$$\frac{f_{k}(t)}{ f_n(t)}= \frac{f_{k}(t)}{ f_{k+1}(t)} \frac{f_{k+1}(t)}{ f_{k+2}(t)}\ldots  \frac{f_{n-1}(t)}{ f_n(t)}$$
and note that the product of monomially positive power series is a monomially positive power series. 
\end{proof}

\begin{lemma}\label{cor1}
Let $t_k(n)=\frac{1}{n}\trace X_n^k$ denote the normalized trace of $X_n^k.$ Then
$t_k(n)-t_k(n-1) \in \R_+[x_1,x_2,\ldots,x_n]$, where $\R_+$ denotes the set of nonnegative real numbers.
\end{lemma}

\begin{proof}
Since $e_n$ is a trace vector for $X_n$ we have
   $$(X_n^k)_{n n} =\frac{1}{n}\trace X_n^k=t_k(n)$$
 and  
\begin{equation}\label{eeq}
 \trace(X_n^k(n))=(1-\frac{1}{n})\trace (X_n^k),
 \end{equation}
where $X_n^k(n)$ denotes the matrix obtained from $X_n^k$ by deleting its last row and column. 
Since $X_{n-1}=X_n(n)$ it is clear that 
$$\trace(X_n^k(n))-\trace (X_{n-1}^k) \in \R_+[x_1,x_2,\ldots,x_n].$$ 
Now (\ref{eeq}) tells us
$$(1-\frac{1}{n})\trace(X_n^k)-\trace (X_{n-1}^k) \in \R_+[x_1,x_2,\ldots,x_n].$$ 
We conclude that 
$$t_k(n)-t_k(n-1)\in \R_+[x_1,x_2,\ldots,x_n],$$
as we wanted to prove.
\end{proof}

\begin{lemma}\label{quot1}
The power series expansion of $w_n(t)=\frac{f_{n-1}(t)}{f_n^{1-1/n}(t)}$ around $t=0$ is monomially positive.
\end{lemma}

\begin{proof}
By Lemma \ref{lem1} we have: 
 \begin{equation}\label{eq1}
 \frac{f_n'(t)}{nf_n(t)}=-\sum_{j=1}^{\infty}t_j(n)t^{j-1}.
 \end{equation}
 Integration of the above equation with respect to $t$ gives us
  \begin{equation}\label{eq2}
  \frac{1}{n}\log f_n(t)=-\sum_{j=1}^{\infty}\frac{t_j(n)t^j}{j}.
  \end{equation}
 Now we have
 \begin{align*}
 \log w_n(t)&=\log f_{n-1}(t)-(1-\frac{1}{n})\log f_{n}(t)\\
                   &=-(n-1)\sum_{j=1}^{\infty}\frac{t_j(n-1)t^j}{j}+(n-1)\sum_{j=1}^{\infty}\frac{t_j(n)t^j}{j}\\
                   &=(n-1)\sum_{j=1}^{\infty}\frac{(t_j(n)-t_j(n-1))t^j}{j}.
 \end{align*}
 Using Lemma \ref{cor1} we conclude that $\log w_n(t)$ is monomially positive. This implies that $w_n(t)=\exp(\log w_n(t))$ is monomially positive. 
\end{proof}

Now we are ready to prove Theorem \ref{main}.

\begin{mainproof}
We will prove the theorem by induction on $n$. Observe that $f_1(t)=1-x_1t$ and the statement is obviously true in this case. Using Lemma \ref{expansion} we get
\begin{align*}
(f_n(t))^{1/n}&=((1-x_1t)f_{n-1}(t))^{1/n}\left(1-\sum_{i=2}^n\frac{(n-1)\ldots(n-i+1)x_if_{n-i}(t)t^i}{(1-x_1t)f_{n-1}(t)}\right)^{1/n}\\
&=((1-x_1t)f_{n-1}(t))^{1/n}(1-V_n(t))^{1/n},
\end{align*}
where 
 $$V_n(t)=\sum_{i=2}^n\frac{(n-1)\ldots(n-i+1)x_if_{n-i}(t)t^i}{(1-x_1t)f_{n-1}(t)}.$$
Notice that Corollary \ref{quot} implies that $V_n(t)$ is monomially positive.
The Taylor expansion of $(1-t)^{1/n}$ around $t=0$ gives us
 $$(1-V_n(t))^{1/n}=1-\sum_{i=1}^{\infty}\alpha_i(V_n(t))^i,$$
 where $\alpha_i >0.$ It follows that: 
  \begin{equation}\label{parts}
  (f_n(t))^{1/n}=((1-x_1t)f_{n-1}(t))^{1/n}-\sum_{i=1}^{\infty}\alpha_i((1-x_1t)f_{n-1}(t))^{1/n}(V_n(t))^i.
  \end{equation}

Now we deal with each term in the above expression separately. 

{\bf Step 1:} We prove that 
$$1-((1- x_1 t)f_{n-1}(t))^{1/n}$$ is monomially positive.

By the induction hypothesis we have 
 $$f_{n-1}(t)^{1/(n-1)}=1-x_1 t-\sum_{j=2}^{\infty}\gamma_j t^j,$$
 where $\gamma_j \in \R_+[x_1,x_2,\ldots,x_n]$.  In addition to induction we also use here that $\gamma_1 = x_1$, which is easy to see.
  \begin{align*}
 ((1-x_1 t)f_{n-1}(t))^{1/n}&=(1- x_1 t)^{1/n}\left((f_{n-1}(t))^{1/(n-1)}\right)^{(n-1)/n}\\
    &=(1-x_1 t)^{1/n}\left(1-x_1 t-\sum_{j=2}^{\infty}\gamma_j t^j\right)^{(n-1)/n}\\
    &=(1-x_1t) \left(1-\frac{\sum_{j=2}^{\infty}\gamma_j t^j}{1-x_1 t}\right)^{(n-1)/n}.
 \end{align*}
 Since 
     $(1-t)^{(n-1)/n}=1-\sum_{k=1}^{\infty}\beta_k t^k,$ where $\beta_k >0,$ 
     we have
     \begin{align*}
 ((1- x_1 t)f_{n-1}(t))^{1/n}&=(1-x_1t) \left(1-\sum_{k=1}^{\infty}\beta_k\left(\frac{\sum_{j=2}^{\infty}\gamma_j t^j}{1-x_1 t}\right)^k\right)\\
 &=1-x_1t-\sum_{k=1}^{\infty}\beta_k\frac{(\sum_{j=2}^{\infty}\gamma_j t^j)^k}{(1-x_1 t)^{k-1}}.
 \end{align*}   
 We finish this step of the proof by observing that $(1-x_1t)^{-(k-1)},$ $k=1,2,\ldots$, is monomially positive.
  
{\bf Step 2} We prove that $$W_n(t)=((1-x_1 t)f_{n-1}(t))^{1/n}V_n(t)$$ is monomially positive. 
 
 \begin{align*}
W_n(t)&= ((1- x_1t)f_{n-1}(t))^{1/n}\left(\sum_{i=2}^n\frac{(n-1)\ldots(n-i+1)x_if_{n-i}(t)}{(1-x_1t)f_{n-1}(t)}t^i\right)\\
     &=\sum_{i=2}^n\frac{(n-1)\ldots(n-i+1)x_i}{(1-x_1t)^{1-1/n}}\cdot \frac{f_{n-i}(t)}{f_{n-2}(t)} \cdot\frac{f_{n-2}(t)}{f_{n-1}(t)^{\frac{n-2}{n-1}}} \cdot \frac{1}{f_{n-1}(t)^{\frac{1}{n(n-1)}}}t^i.        
  \end{align*}
Corollary \ref{quot} tells us that $$\frac{f_{n-i}(t)}{f_{n-2}(t)}$$ is monomially positive.
Lemma 
\ref{quot1} tells us that $$\frac{f_{n-2}(t)}{f_{n-1}(t)^{\frac{n-2}{n-1}}}$$ is monomially positive.
From the induction hypothesis we get that
$$1-f_{n-1}(t)^{\frac{1}{n-1}}$$ is monomially positive, and since $1-(1-x)^{\alpha}$ has positive coefficients for $\alpha \in (0,1)$, we conclude that $$1-f_{n-1}(t)^{\frac{1}{n(n-1)}}$$ is monomially positive. This implies that $$ \frac{1}{f_{n-1}(t)^{\frac{1}{n(n-1)}}}$$ is monomially positive.

{\bf Step 3:} We prove that $$((1-x_1t)f_{n-1}(t))^{1/n}(V_n(t))^k$$ is monomially positive.
 
 We have
 \begin{align*}
 ((1-x_1 t)f_{n-1}(t))^{1/n}(V_n(t))^k&=(((1- x_1 t)f_{n-1}(t))^{1/n}V_n(t))(V_n(t))^{k-1}\\
  &=W_n(t)(V_n(t))^{k-1}.
 \end{align*}
 Since $V_n(t)$ is monomially positive so is $(V_n(t))^{k-1}$, and we have already proved that $W_n(t)$ is monomially positive.
\end{mainproof}

\section{Sign of a determinant}

In this final section we show that the positivity of coefficients established in Theorem \ref{main} is equivalent to the positivity of a certain class of determinants. 

Let 
 $$s_k(n)=\trace X_n^k$$
denote the trace of $X_n^k$ and 
$$t_k(n)=\frac{1}{n}\trace X_n^k$$ denote the normalized traces of $X_n^k.$ 
We define the following matrix
\begin{equation}\label{Tn}
T_m(n)=\npmatrix{t_1(n) & 1 & 0 &  0 & \ldots & 0 \\
                            t_2(n) & t_1(n) & 2 & 0 & \ldots & 0 \\
                            t_3(n) & t_2(n) & t_1(n) & 3 & \ddots & \vdots \\ 
                            \vdots & \ddots & \ddots & \ddots & \ddots &0 \\
                            t_{m-1}(n) & \ldots & t_3(n)  & t_2(n) & t_1(n) & (m-1) \\
                            t_m(n)  & \ldots & t_{4}(n) & t_3(n) & t_2(n) & t_1(n)}.
\end{equation}

\begin{theorem}\label{determinants}
Let $f_n(t)=\det(I_n-tX_n)$ and let
 $$h_n(t)=f_n(t)^{1/n}=1-\sum_{j=1}^{\infty}\gamma_j(n) t^j.$$
Then
$$\gamma_m(n)=\frac{1}{m!}(-1)^{m-1} \det T_m(n).$$
\end{theorem}

\begin{proof}
First we compute 
\begin{align*}
h_n'(t)=\frac{1}{n}(f_n(t))^{(1/n)-1}f_n'(t).
\end{align*}
Now we use Lemma \ref{lem1} to get
\begin{align*}
\frac{h_n'(t)}{h_n(t)}&=\frac{1}{n}\frac{f_n'(t)}{f_n(t)}\\
 &=\frac{1}{n}(-\sum_{j=1}^{\infty} s_j(n)t^{j-1}) \\
 &=-\sum_{j=1}^{\infty} t_j(n)t^{j-1}.
\end{align*}
Let $$h_n(t)=1-\sum_{j=1}^{\infty}\gamma_j(n) t^j.$$ Theorem \ref{main} tells us that  $\gamma_j(n) \in \R_+[x_1,x_2,\ldots,x_n].$
Comparing coefficients of $t$ in the equation
 $$\sum_{j=1}^{\infty}\gamma_j(n) jt^{j-1}=(1-\sum_{j=1}^{\infty}\gamma_j(n) t^{j})(\sum_{j=1}^{\infty} t_j(n)t^{j-1})$$
 we get Newton identities:
 \begin{align}
j \gamma_j(n)=t_j(n)-\sum_{i=1}^{j-1}\gamma_i(n)t_{j-i}(n)
 \end{align}
that can be written in the matrix form:
$$
\npmatrix{1 & 0 &  0 & \ldots & 0 \\
                    t_1(n) & 2 & 0 & \ldots & 0 \\
                     t_2(n) & t_1(n) & 3 & \ddots & \vdots \\ 
                      \vdots & \ddots & \ddots & \ddots &0 \\
                       t_{m-1}(n) & \ldots  & t_2(n) & t_1(n) & m}
\npmatrix{\gamma_1(n) \\ \gamma_2(n) \\ \gamma_3(n) \\ \vdots \\ \gamma_m(n)}=
\npmatrix{t_1(n) \\ t_2(n) \\ t_3(n) \\ \vdots \\ t_m(n)}.
$$
We use Cramer's rule to find $\gamma_m(n):$
\begin{align*}
\gamma_m(n) &=\frac{1}{m!} \det \npmatrix{1 & 0 &  0 & \ldots & 0 & t_1(n) \\
                    t_1(n) & 2 & 0 & \ldots & 0 & t_2(n)\\
                     t_2(n) & t_1(n) & 3 & \ddots & \vdots & t_3(n) \\ 
                      \vdots & \ddots & \ddots & \ddots & m-1 & \vdots\\
                       t_{m-1}(n) & \ldots  & t_3(n) & t_2(n) & t_1(n) & t_m(n)} \\
                    &= \frac{1}{m!}(-1)^{m-1} \det T_m(n). 
                    \end{align*}
The argument involving Cramer's Rule in the context of the Newton identities is due to Brioschi. Such identities can be found in \cite{MR901944}.                    \end{proof}
                    
\begin{corollary}
 $(-1)^{m-1}\det(T_m(n)) \in \R_+[x_1,x_2,\ldots,x_n]$ for all positive integers $m$ and $n.$ 
\end{corollary}      

\begin{proof}
Corollary of Theorem \ref{main} and Theorem \ref{determinants}.
\end{proof}         
                
 \section{Concluding Remarks}   
 
 In this paper we have proved a positivity result on power series  occurring in the  study of the NIEP. In the proof we used properties of a special patterned matrix $X_n$ defined in (\ref{Sn}). The matrix $X_n$ has interesting combinatorial properties which will be presented in our future work.         
                            
The statement of Theorem \ref{NIEPappl} purely involves hypothesis on the polynomials and their roots. This suggests that it should be possible to prove the theorem using methods of complex analysis without recourse to matrix theory, but we have not succeeded in doing this so far.    

\bibliographystyle{plain}
\bibliography{NIEP}

\end{document}